\documentclass[reqno,10pt,centertags]{amsart}
\usepackage{amsmath,amsthm,amscd,amssymb,latexsym,esint,upref,stmaryrd,
enumerate,color,verbatim,yfonts,mathrsfs} 

\usepackage{hyperref}

\newcommand*{\mailto}[1]{\href{mailto:#1}{\nolinkurl{#1}}}

\newcommand{\bbC}{{\mathbb{C}}}

\newcommand{\bbN}{{\mathbb{N}}}

\newcommand{\bbR}{{\mathbb{R}}}

\newcommand{\cB}{{\mathcal B}}

\newcommand{\cF}{{\mathcal F}}

\newcommand{\cM}{{\mathcal M}}

\newcommand{\cS}{{\mathcal S}}

\newcommand{\gM}{{\mathfrak{M}}}

\DeclareMathOperator{\supp}{supp}

\DeclareMathOperator{\dom}{dom}

\renewcommand{\Re}{\text{\rm Re}}

\newcommand{\loc}{\text{\rm{loc}}}

\newcommand{\no}{\notag}
\newcommand{\lb}{\label}
\newcommand{\f}{\frac}

\newcommand{\ol}{\overline}

\newcommand{\bi}{\bibitem}

\let\geq\geqslant
\let\leq\leqslant

\makeatletter
\def\theequation{\@arabic\c@equation}

\allowdisplaybreaks 
\numberwithin{equation}{section}

\newtheorem{theorem}{Theorem}[section]

\newtheorem{lemma}[theorem]{Lemma}
\newtheorem{corollary}[theorem]{Corollary}
\newtheorem{definition}[theorem]{Definition}

\newtheorem{example}[theorem]{Example}

\theoremstyle{remark}
\newtheorem{remark}[theorem]{Remark}

\begin{document}

\numberwithin{equation}{section}
\allowdisplaybreaks

\title[Positivity Preserving, Translation Invariant, Operators in 
$L^p(\bbR^n)^m$]{On Positivity Preserving, Translation Invariant, Operators in 
$L^p(\bbR^n)^m$} 
    
\author[F.\ Gesztesy]{Fritz Gesztesy}
\address{Department of Mathematics, 
Baylor University, One Bear Place \#97328,
Waco, TX 76798-7328, USA}
\email{\mailto{Fritz\_Gesztesy@baylor.edu}}
\urladdr{\url{http://www.baylor.edu/math/index.php?id=935340}}

\author[M.\ M.\ H.\ Pang]{Michael M.\ H.\ Pang}
\address{Department of Mathematics,
University of Missouri, Columbia, MO 65211, USA}
\email{\mailto{pangm@missouri.edu}}
\urladdr{\url{https://www.math.missouri.edu/people/pang}}

\dedicatory{Dedicated, with admiration, to the memory of Boris Pavlov (1936--2016)}


\date{\today}

\begin{abstract} 
We characterize positivity preserving, translation invariant, linear operators in $L^p(\bbR^n)^m$, 
$p \in [1,\infty)$, $m,n \in \bbN$. 
\end{abstract}

\maketitle 


\section{Introduction}  \lb{s1}

This note should be viewed as an addendum to our paper \cite{GP17}, which was devoted to (conditional) positive semidefiniteness of matrix-valued functions and positivity preserving operators on spaces of matrix-valued functions. In the present note we consider the case of vector-valued functions.  

More precisely, recalling that $F \colon \bbR^n \to \bbC$ is called {\bf positive semidefinite} if for all 
$N \in \bbN$, $x_p \in \bbR^n$, $1 \leq p \leq N$, the matrix $(F(x_p - x_q))_{1 \leq p,q \leq N} \in \bbC^{N \times N}$ is positive semidefinite, the principal result proved in this note on positivity preserving, translation invariant, linear operators in $L^p(\bbR^n)^m$, $p \in [1,\infty)$, $m,n \in \bbN$, reads as follows:

\begin{theorem} \lb{t1.1} 
Let $G\colon \bbR^n \to \bbC^{m \times m}$ be bounded and continuous. Then the following items 
$(i)$--$(iv)$ are equivalent: \\[1mm]
$(i)$ For all $1 \leq j,k \leq m$, $G_{j,k} \colon \bbR^n \to \bbC$ is positive semidefinite. \\[1mm]
$(ii)$ The linear operator, 
\begin{equation} 
G(- i \nabla)\big|_{C_0^{\infty}(\bbR^n)^m} \colon 
\begin{cases} C_0^{\infty}(\bbR^n)^m \to 
C_{\infty}(\bbR^n)^m, \\
(f_1,\dots,f_m)^\top \mapsto 
\Big(G\big(f_1^\wedge,\dots,f_m^{\wedge}\big)^\top\Big)^{\vee},
\end{cases}       \lb{1.1} 
\end{equation} 
extends boundedly to $G(- i \nabla) \in \cB\big(L^1(\bbR^n)^m\big)$ 
satisfying
\begin{equation}
G(- i \nabla) \big(L^1(\bbR^n)_+^m\big) \subseteq L^1(\bbR^n)_+^m. 
\end{equation}
$(iii)$ There exists $p \in (1,\infty)$ such that the linear operator \eqref{1.1} extends  boundedly to $G(- i \nabla) \in 
\cB\big(L^p(\bbR^n)^m\big)$ 
satisfying
\begin{equation}
G(- i \nabla) \big(L^p(\bbR^n)_+^m\big) \subseteq L^p(\bbR^n)_+^m. 
\end{equation}
$(iv)$ For all $p \in (1,\infty)$, the linear operator \eqref{1.1} extends boundedly to $G(- i \nabla) \in \cB\big(L^p(\bbR^n)^m\big)$ 
satisfying
\begin{equation}
G(- i \nabla) \big(L^p(\bbR^n)_+^m\big) \subseteq 
L^p(\bbR^n)_+^m. 
\end{equation}
\end{theorem}

Here $L^p(\bbR^n)_+^m$ denotes the set of elements 
$f = (f_1,\dots,f_m)^{\top} \in L^p(\bbR^n)^m$ such that $f_k \geq 0$ (Lebesgue) a.e., 
$1 \leq k \leq m$. 

To make this note somewhat self-contained, we summarize in Appendix \ref{sA} the necessary background material on (conditionally) positive semidefinite functions $F \colon \bbR^n \to \bbC$ 
(hinting briefly at some matrix-valued generalizations), 
and on Fourier multipliers in $L^1(\bbR^n)$ and $L^2(\bbR^n)$.

\smallskip

Finally, we briefly summarize the basic notation employed in this paper: The Banach space of bounded linear operators on a complex Banach space $X$ is denoted by $\cB(X)$. 

For $Y$ a set, $Y^m$, $m \in \bbN$, represents the set of 
$m \times 1$ matrices with entries in $Y$; similarly, $Y^{m \times n}$, 
$m,n \in \bbN$, represents 
the set of $m \times n$ matrices with entries in $Y$.      

Unless explicitly stated otherwise, $\bbC^m$ is always equipped with the Euclidean scalar 
product $( \, \cdot \, , \, \cdot \,)_{\bbC^m}$ and associated norm $\|\, \cdot \, \|_{\bbC^m}$.  

The symbol $\cS(\bbR^n)$ denotes the standard Schwartz space of all complex-valued 
rapidly decreasing functions on $\bbR^n$ . In addition, we employ the spaces,
\begin{align} 
& C_0^{\infty}(\bbR^n) = \{f \in C^{\infty}(\bbR^n) \, | \, \supp\,(f) \, \text{compact}\},    \lb{1.10} \\
& C_{\infty}(\bbR^n) = \{f \in C(\bbR^n) \, | \, \lim_{|x| \to \infty} f(x) = 0 \}. 
\lb{1.12}
\end{align}
Unless explicitly stated otherwise, the spaces \eqref{1.10}--\eqref{1.12} are always equipped with the 
norm $\|f\|_{\infty} = {\rm ess.sup}_{x \in \bbR^n} | f(x) |$. 

For brevity, we will omit displaying the Lebesgue measure $d^n x$ in 
$L^p(\bbR^n)$, $p \in [1,\infty)\cup \{\infty\}$, whenever the latter is understood. The norm for 
$ f = (f_1,\dots,f_m)^\top \in L^p(\bbR^n)^m$, $ p \in [1,\infty)$, $m \in \bbN$, is defined by
\begin{equation}
\|f\|_{L^p(\bbR^n)^m} = \sum_{j=1}^m \|f_j\|_{L^p(\bbR^n)}, \quad 
 f = (f_1,\dots,f_m)^\top \in L^p(\bbR^n)^m. 
\end{equation}
The symbols $L^p(\bbR^n)_+^m$ \big(resp., $C_0^{\infty}(\bbR^n)_+^m$\big) 
represent elements of $L^p(\bbR^n)^m$ \big(resp., $C_0^{\infty}(\bbR^n)^m$\big) 
with all entries nonnegative (Lebesgue) a.e.

The Fourier and inverse Fourier transforms on $\cS(\bbR^n)$ are denoted by the pair of formulas,
\begin{align}
(\cF f)(y) = f^{\wedge}(y) &= (2 \pi)^{-n/2} \int_{\bbR^n} e^{- i (y \cdot x)} f(x) \, d^n x,    \lb{1.15} \\
(\cF^{-1} g)(x) = g^{\vee}(x) &= (2 \pi)^{-n/2} \int_{\bbR^n} e^{i (x \cdot y)} g(y) \, d^n y, \lb{1.16} \\
& \hspace*{2.75cm}  f, g \in \cS(\bbR^n),    \no 
\end{align}
and we use the same notation for the appropriate extensions, where 
$\cS(\bbR^n)$ is replaced by $L^p(\bbR^n)$, $p \in [1,\infty)$. 

The open ball in $\bbR^n$ with center $x_0 \in \bbR^n$ and radius $r_0 > 0$ is denoted by 
the symbol $B_n(x_0, r_0)$, the norm of vectors $x \in \bbR^n$ is denoted by $\|x\|_{\bbR^n}$, the 
scalar product of $x, y \in \bbR^n$, is abbreviated by $(x, y)_{\bbR^n}$.  

With $\gM_n$ the $\sigma$-algebra of all Lebesgue measurable subsets 
of $\bbR^n$ and for $E \in \gM_n$, 
the $n$-dimensional Lebesgues measure of $E$ is abbreviated by $|E|$.

\section{On Positivity Preserving Linear Operators in $L^p(\bbR^n)^m$} \lb{s2}

In this section we characterize linear, positivity preserving, translation invariant operators in $L^p(\bbR^n)^m$, $p \in [1,\infty)$, $m,n \in \bbN$. 

For basic notions and conventions used we refer to the background material in Appendix \ref{sA}. 

We start with the following result: 

\begin{lemma} \lb{l2.1}  
Let $n \in \bbN$, and suppose that $F_{\ell} \colon \bbR^n \to \bbC$, $\ell =1,2$, are positive 
semidefinite. Then $F_1 F_2:\bbR^n \to \bbC$ is positive semidefinite. 
\end{lemma}
\begin{proof}
Let $N \in \bbN$, $x_p \in \bbR^n$, $1 \leq p \leq N$, then by hypothesis, 
$(F_{\ell}(x_p - x_q))_{1 \leq p,q \leq N} \geq 0$, $\ell = 1,2$, and hence by Schur's theorem 
(see, e.g., the Lemma in \cite[p.~215]{RS78})
\begin{align}
\begin{split} 
& (F_1F_2)(x_p - x_q))_{1 \leq p,q \leq N}    \\
& \quad = (F_1(x_p - x_q))_{1 \leq p,q \leq N} \circ_H (F_2(x_p - x_q))_{1 \leq p,q \leq N} \geq 0. 
\end{split} 
\end{align}
Here $A \circ_H B$ denotes the {\it Hadamard product} of two matrices $A, B \in \bbC^{N \times N}$, 
defined by 
\begin{equation}
(A \circ_H B)_{j,k} = A_{j,k} B_{j,k}, \quad 1 \leq j, k \leq N. 
\end{equation}
\end{proof}

The principal result on positivity preserving, translation invariant, linear operators in $L^p(\bbR^n)^m$, $p \in [1,\infty)$, $m,n \in \bbN$, proved in this note then reads as follows:

\begin{theorem} \lb{t2.2} 
Let $m, n \in \bbN$, and $G\colon \bbR^n \to \bbC^{m \times m}$ be bounded and continuous. Then the following items 
$(i)$--$(iv)$ are equivalent: \\[1mm]
$(i)$ For all $1 \leq j,k \leq m$, $G_{j,k} \colon \bbR^n \to \bbC$ is positive semidefinite. \\[1mm]
$(ii)$ The linear operator, 
\begin{equation} 
G(- i \nabla)\big|_{C_0^{\infty}(\bbR^n)^m} \colon 
\begin{cases} C_0^{\infty}(\bbR^n)^m \to 
C_{\infty}(\bbR^n)^m, \\
(f_1,\dots,f_m)^\top \mapsto 
\Big(G\big(f_1^\wedge,\dots,f_m^{\wedge}\big)^\top\Big)^{\vee},
\end{cases}       \lb{2.1} 
\end{equation} 
extends boundedly to $G(- i \nabla) \in \cB\big(L^1(\bbR^n)^m\big)$ 
satisfying
\begin{equation}
G(- i \nabla) \big(L^1(\bbR^n)_+^m\big) \subseteq L^1(\bbR^n)_+^m. 
\end{equation}
$(iii)$ There exists $p \in (1,\infty)$ such that the linear operator \eqref{2.1} extends  boundedly to $G(- i \nabla) \in \cB\big(L^p(\bbR^n)^m\big)$ 
satisfying
\begin{equation}
G(- i \nabla) \big(L^p(\bbR^n)_+^m\big) \subseteq L^p(\bbR^n)_+^m. 
\end{equation}
$(iv)$ For all $p \in (1,\infty)$, the linear operator \eqref{2.1} extends boundedly to  
$G(- i \nabla) \in \cB\big(L^p(\bbR^n)^m\big)$ 
satisfying
\begin{equation}
G(- i \nabla) \big(L^p(\bbR^n)_+^m\big) \subseteq L^p(\bbR^n)_+^m. 
\end{equation}
\end{theorem}
\begin{proof}
We begin by proving the equivalence of items $(i)$ and $(ii)$. 

First, suppose $(i)$ holds. By 
Bochner's theorem (cf.\ Theorem \ref{tA.3}), there exist finite, nonnegative Borel measures $\mu_{j,k}$ on 
$\bbR^n$ such that $G_{j,k} = \mu_{j,k}^{\wedge}$, $1 \leq j,k \leq m$. Thus, the classical 
$L^1(\bbR^n)$-multiplier theorem (cf.\ Theorem \ref{tA.7}\,$(i)$) implies that $G_{j,k}(- i \nabla)$ is an 
$L^1(\bbR^n)$ multiplier operator for all $1 \leq j,k \leq m$, that is, $G_{j,k}(- i \nabla) \in \cM^{1,1}(\bbR^n)$, 
$1 \leq j,k \leq m$. With $\|G_{j,k}(- i \nabla)\|_{1,1}$ denoting the operator norm of $G_{j,k}(- i \nabla)$ in $L^1(\bbR^n)$, $1 \leq j,k \leq m$, we introduce 
\begin{equation}
\|G(- i \nabla)\|_{1,1} := \max_{1 \leq j,k \leq m} \{\|G_{j,k}(- i \nabla)\|_{1,1}\}. 
\end{equation}
Then, for $f = (f_1,\dots,f_m)^{\top} \in C_0^{\infty}(\bbR^n)^m$, 
\begin{equation}
\big\|\big[G(- i \nabla)\big|_{C_0^{\infty}(\bbR^n)^m} f\big]_j\big\|_{L^1(\bbR^n)} 
\leq \|G(- i \nabla)\|_{1,1} \|f\|_{L^1(\bbR^n)^m}, \quad 1 \leq j \leq m, 
\end{equation}
and hence,
\begin{equation}
\big\|G(- i \nabla)\big|_{C_0^{\infty}(\bbR^n)^m} f\big\|_{L^1(\bbR^n)^m} 
\leq m \|G(- i \nabla)\|_{1,1} \|f\|_{L^1(\bbR^n)^m}.   
\end{equation}
In other words, $G(- i \nabla)\big|_{C_0^{\infty}(\bbR^n)^m}$ extends boundedly to an operator 
$G(- i \nabla) \in \cB\big(L^1(\bbR^n)^m\big)$. 

Since by \cite[Theorem~2.29\,(c)]{AF03}, $C_0^{\infty}(\bbR^n)_+^m$ is dense in 
$L^1(\bbR^n)_+^m$, to show that item $(ii)$ holds, it suffices to prove that 
\begin{equation}
G(- i \nabla)\big(C_0^{\infty}(\bbR^n)_+^m\big) \subseteq L^1(\bbR^n)_+^m. 
\end{equation}
Let $f = (f_1,\dots,f_m)^{\top} \in C_0^{\infty}(\bbR^n)_+^m$. By Bochner's theorem 
(cf.\ Theorem \ref{tA.3}, with measures of the form $\mu_j = f_j d^n x$, $1 \leq j \leq m$), $f_k^{\wedge}$, 
$1 \leq k \leq m$, are positive semidefinite, and thus an application of Lemma \ref{l2.1} yields that 
$\sum_{k=1}^m G_{j,k} f_k^{\wedge}$, $1 \leq j \leq m$, are positive semidefinite. Since 
\begin{equation}
([G(- i \nabla) f]_j)^{\wedge} = \sum_{k=1} G_{j,k} f_k^{\wedge}, \quad 1 \leq j \leq m, 
\end{equation}
applying Bochner's theorem once more implies that $[G(- i \nabla) f]_j \geq 0$, $1 \leq j \leq m$, 
that is,
\begin{equation}
G(- i \nabla) f \in L^1(\bbR^n)_+^m. 
\end{equation}
Thus, item $(ii)$ holds.  

Next, we assume item $(ii)$ holds. Let $\varphi \colon [0,\infty) \to [0,\infty)$ satisfy 
\begin{align}
&(\alpha) \;\, \varphi(\, \cdot \,) \, \text{ is decreasing on $[0,\infty)$,}   \no  \\
& (\beta) \;\, \varphi \in C^{\infty}([0,\infty)), \quad \varphi^{(k)}(0) = 0, \; k \in \bbN, \quad \supp(\varphi) = [0,1],    \\
&(\gamma) \;\, \int_{\bbR^n} d^n x \, \phi(x) =1, \quad 
\phi(x) = \varphi(\|x\|_{\bbR^n}), \; x \in \bbR^n,   \no 
\end{align}
and introduce
\begin{equation}
\phi_{\varepsilon}(x) = \varepsilon^{-n} \phi(x/\varepsilon), \quad \varepsilon \in (0,1), \; x \in \bbR^n,
\end{equation}
implying
\begin{equation}
\lim_{\varepsilon \downarrow 0} \phi_{\varepsilon}^{\wedge}(\xi) = 1, \quad \xi \in \bbR^n.   \lb{2.15} 
\end{equation}
Moreover we introduce
\begin{equation}
\phi_{\varepsilon,k} \in L^1(\bbR^n)_+^m, \quad [\phi_{\varepsilon,k}]_j = 
\begin{cases} 0, & j \neq k, \\
\phi_{\varepsilon}, & j=k, 
\end{cases} \quad 1 \leq j, k \leq m, \; \varepsilon \in (0,1).    \lb{2.16} 
\end{equation}
Then by hypothesis,
\begin{equation}
G(- i \nabla) \phi_{\varepsilon,k} = \Big(\big(G_{1,k} \phi_{\varepsilon}^{\wedge}\big)^{\vee},
\dots, \big(G_{m,k} \phi_{\varepsilon}^{\wedge}\big)^{\vee}\Big)^{\top} \in 
L^1(\bbR^n)_+^m,
\end{equation}
and hence, once more by Bochner's theorem, $G_{j,k} \phi_{\varepsilon}^{\wedge}$, $1 \leq j,k \leq m$,
are positive semidefinite. Thus, for all $x_p \in \bbR^n$, $1 \leq p \leq N$, \eqref{2.15} implies 
\begin{equation}
0 \leq \lim_{\varepsilon \downarrow 0} \big(G_{j,k}(x_p - x_q) \phi_{\varepsilon}^{\wedge} (x_p - x_q)
\big)_{1 \leq p,q \leq N} = (G_{j,k}(x_p - x_q))_{1 \leq p,q \leq N},
\end{equation}
that is, $G_{j,k}$, $1 \leq j,k \leq m$, are positive semidefinite, implying item $(i)$.  

Next we prove that item $(ii)$ implies item $(iv)$. Assuming item $(ii)$ holds, then by the equivalence of items $(i)$ and $(ii)$ just proved, $G_{j,k}$, $1 \leq j,k \leq m$, are positive semidefinite and hence there exist finite, nonnegative Borel measures $\mu_{j,k}$ on $\bbR^n$, such that $G_{j,k} = \mu_{j,k}^{\wedge}$, $1 \leq j,k \leq m$. By the classical $L^1$-multiplier theorem (cf.\ 
Theorem \ref{tA.7}\,$(i)$), this implies that $G_{j,k}(- i \nabla)$, $1 \leq j,k \leq m$, are $L^1(\bbR^n)$ multiplier operators, and hence also $L^p(\bbR^n)$-multipliers for all $p \in [1,\infty)$ according to \cite[p.~143, remarks after Definition~2.5.11]{Gr08}. We denote by $\|G_{j,k}(- i \nabla)\|_{p,p}$ the norm of $G_{j,k}(- i \nabla)$ in $L^p(\bbR^n)$, $1 \leq j,k \leq m$, $p \in (1,\infty)$, and introduce 
\begin{equation}
\|G(- i \nabla)\|_{p,p} := \max_{1 \leq j,k \leq m} \{\|G_{j,k}(- i \nabla)\|_{p,p}\}. 
\end{equation}
Then for all $f = (f_1,\dots, f_m)^{\top} \in C_0^{\infty}(\bbR^n)^m$, 
\begin{equation}
\big\|G(- i \nabla)\big|_{C_0^{\infty}(\bbR^n)^m} f\big\|_{L^p(\bbR^n)^m} \leq 
m \|G(- i \nabla)\|_{p,p} \|f\|_{L^p(\bbR^n)^m}, 
\end{equation}
and thus $G(- i \nabla)\big|_{C_0^{\infty}(\bbR^n)^m}$ can be extended to a bounded operator 
$G(- i \nabla) \in \cB\big(L^p(\bbR^n)^m\big)$. 

Since once more by \cite[Theorem~2.29\,(c)]{AF03}, $C_0^{\infty}(\bbR^n)_+^m$ is dense in 
$L^p(\bbR^n)_+^m$, $p \in [1,\infty)$, to show that item $(iv)$ is valid it suffices to prove that 
\begin{equation}
[G(- i \nabla) f]_j \geq 0, \quad 1 \leq j \leq m, \; f =(f_1,\dots, f_m)^{\top} \in 
C_0^{\infty}(\bbR^n)_+^m, 
\end{equation}
which is implied by item $(ii)$. 

It is clear that item $(iv)$ implies item $(iii)$.

Next, we prove that item $(iii)$ implies item $(i)$. We start by proving that for all 
$f =(f_1,\dots, f_m)^{\top} \in C_0^{\infty}(\bbR^n)_+^m$, and all $1 \leq j \leq m$, 
$\sum_{k=1}^m G_{j,k} f_k^{\wedge}$ is positive semidefinite. Since $G(\, \cdot \,)$ is bounded, the map
\begin{equation}
f \mapsto G(- i \nabla) f = \big(G f^{\wedge}\big)^{\vee}, \quad f \in C_0^{\infty}(\bbR^n)^m,
\end{equation}
extends to a bounded operator in $\cB\big(\big(L^2(\bbR^n)\big)^m\big)$ by the classical 
$L^2$-multiplier theorem, Theorem \ref{tA.7}\,$(ii)$. Thus, for all $R > 0$, 
$ f =(f_1,\dots, f_m)^{\top} \in C_0^{\infty}(\bbR^n)_+^m$, and all $1 \leq j \leq m$, 
$[G(- i \nabla)f]_j \chi_{B_n(0,R)} \in L^1(\bbR^n)_+ \cap L^2(\bbR^n)_+$. Thus, again by Bochner's theorem, $([G(- i \nabla) f]_j \chi_{B_n(0,R)})^{\wedge}$, $1 \leq j \leq m$, are positive semidefinite. Hence, taking limits in $L^2(\bbR^n)$, one obtains,
\begin{equation}
\sum_{k=1}^m G_{j,k} f_k^{\wedge} = ([G(- i \nabla) f]_j)^{\wedge} = 
\lim_{R \uparrow \infty} ([G(- i \nabla) f]_j \chi_{B_n(0,R)})^{\wedge}, 
\end{equation}
and thus there exist a sequence of increasing positive numbers $\{R_{\ell}\}_{\ell \in \bbN}$, with 
$\lim_{\ell \to \infty} R_{\ell} = \infty$, and a set $E \subset \bbR^n$ of Lebesgue measure zero, $|E|=0$, such that for all $x \in \bbR^n \backslash E$,
\begin{equation}
\lim_{\ell \to \infty} ([G(- i \nabla) f]_j \chi_{B_n(0,R_{\ell})})^{\wedge} (x) = 
\sum_{k=1}^m G_{j,k}(x) f_k^{\wedge}(x).
\end{equation}
Letting $x_p \in \bbR^n$, $1 \leq p \leq N$, for each $1 \leq q \leq N$, one can choose a sequence 
$\{x_{q,r}\}_{r \in \bbN} \subset \bbR^n \backslash E$ such that $\lim_{r \to \infty} x_{q,r} = x_q$, 
$1 \leq q \leq N$, and that
\begin{equation}
(x_{p,r} - x_{q,r}) \in \bbR^n \backslash E, \quad 1 \leq p, q \leq N, \; r \in \bbN.
\end{equation}
Hence, employing that $\sum_{k=1}^m G_{j,k} f_k^{\wedge}$, $1 \leq j \leq m$, is continuous, one concludes that
\begin{align}
& \bigg(\sum_{k=1}^m G_{j,k}(x_p - x_q) f_k^{\wedge}(x_p - x_q)\bigg)_{1 \leq p, q \leq N}   \no \\
& \quad = \lim_{r \to \infty} \bigg(\sum_{k=1}^m G_{j,k}(x_{p,r} - x_{q,r}) 
f_k^{\wedge}(x_{p,r} - x_{q,r})\bigg)_{1 \leq p, q \leq N}   \no \\
& \quad = \lim_{r \to \infty} \lim_{\ell \to \infty} \big(([G(- i \nabla) f]_j 
\chi_{B_n(0,R_{\ell}})^{\wedge}(x_{p,r} - x_{q,r})\big)_{1 \leq p, q \leq N}  \no \\
& \quad \geq 0,    \lb{2.26} 
\end{align}
implying that $\sum_{k=1}^m G_{j,k} f_k^{\wedge}$, $1 \leq j \leq m$, are positive semidefinite. 
Next, introduce $\phi_{\varepsilon,k} \in C_0^{\infty}(\bbR^n)_+^m$, $1 \leq k \leq m$, 
$\varepsilon \in (0,1)$, as in 
\eqref{2.16}. Then with $f = \phi_{\varepsilon,k}$, $1 \leq k \leq m$, $\varepsilon \in (0,1)$, in 
\eqref{2.26}, one concludes that for all $\varepsilon \in (0,1)$, $1 \leq j,k \leq m$, 
$G_{j,k} \phi_{\varepsilon}^{\wedge}$ is positive semidefinite. Hence, again applying \eqref{2.15}, one obtains that $G_{j,k} = \lim_{\varepsilon \downarrow 0} G_{j,k} \phi_{\varepsilon}^{\wedge}$ is positive semidefinite, and thus item $(i)$ holds. 
\end{proof}

Given $S \in \bbC^{m \times m}$, $m \in \bbN$, its {\it Hadamard exponential}, denoted 
by $\exp_H(S)$, is defined by 
\begin{equation}
\exp_H(S) = \big(\exp_H(S)_{j,k} := \exp(S_{j,k})\big)_{1 \leq j,k \leq m}.  
\end{equation}
More generally, if $S \colon \bbR^n \to \bbC^{m \times m}$, $m,n \in \bbN$, its {\it Hadamard exponential}, denoted by $\exp_H(S(\, \cdot \,))$, is defined by 
\begin{equation}
\exp_H(S(x)) = \big(\exp_H(S(x))_{j,k} := \exp(S(x)_{j,k})\big)_{1 \leq j,k \leq m}, \quad x \in \bbR^n.  
\end{equation}

\begin{corollary} \lb{c2.3}
Let $m, n \in \bbN$, suppose that $F \colon \bbR^n \to \bbC^{m \times m}$ is continuous, and assume there exists $c \in \bbR$ such that 
\begin{equation}
\Re(F_{j,k}) \leq c, \quad 1 \leq j,k \leq m.
\end{equation}
Then the following items $(i)$--$(vi)$ are equivalent: \\[1mm]
$(i)$ There exists $p \in (1,\infty)$ such that for all $t > 0$, the linear operator
\begin{equation}
(\exp_H(tF))(- i \nabla)\big|_{C_0^{\infty}(\bbR^n)^m} \colon \begin{cases} 
C_0^{\infty}(\bbR^n)^m \to C_{\infty}(\bbR^n)^m, \\
(f_1,\dots,f_m)^{\top} \mapsto \Big(\exp_H(tF) \big(f_1^{\wedge}, \dots, f_m^{\wedge}\big)^{\top}\Big)^{\vee},
\end{cases}     \lb{2.28}
\end{equation}
extends to a bounded operator $(\exp_H(tF))(- i \nabla) \in \cB\big(L^p(\bbR^n)^m\big)$ satisfying 
\begin{equation}
(\exp_H(tF))(- i \nabla) \big(L^p(\bbR^n)_+^m\big) \subseteq L^p(\bbR^n)_+^m.  
\end{equation}
$(ii)$ For all $p \in (1,\infty)$, and all $t > 0$, the linear operator \eqref{2.28} 
extends boundedly to $(\exp_H(tF))(- i \nabla) \in \cB\big(L^p(\bbR^n)^m\big)$ satisfying 
\begin{equation}
(\exp_H(tF))(- i \nabla) \big(L^p(\bbR^n)_+^m\big) \subseteq L^p(\bbR^n)_+^m.  
\end{equation}
$(iii)$ For all $t > 0$, the linear operator \eqref{2.28} 
extends boundedly to $(\exp_H(tF))(- i \nabla)$ $\in \cB\big(L^1(\bbR^n)^m\big)$ satisfying 
\begin{equation}
(\exp_H(tF))(- i \nabla) \big(L^1(\bbR^n)_+^m\big) \subseteq L^1(\bbR^n)_+^m.  
\end{equation}
$(iv)$ For all $1 \leq j, k \leq m$, and all $t > 0$, $\exp(t F_{j,k}) \colon \bbR^n \to \bbC$ is 
positive semidefinite. \\[1mm] 
$(v)$ For all $1 \leq j, k \leq m$, $F_{j,k} \colon \bbR^n \to \bbC$ is conditionally 
positive semidefinite. \\[1mm] 
$(vi)$ For all $1 \leq j,k \leq m$, there exist $\alpha_{j,k} \in \bbR$, $\beta_{j,k} \in \bbR^n$, 
$0 \leq A(j,k) \in \bbC^{n \times n}$, and a nonnegative, finite Borel measure $\nu_{j,k}$ on $\bbR^n$, satisfying $\nu_{j,k}(\{0\}) = 0$, such that 
\begin{align}
F_{j,k}(x) &= \alpha_{j,k} + i (\beta_{j,k}, x)_{\bbR^n} - (x, A(j,k) x))_{\bbC^n}   \\
& \quad + \int_{\bbR^n} \bigg[\exp(i (x, y)_{\bbR^n}) - 1 - \f{i (x,y)_{\bbR^n}}{1 + \|y\|_{\bbR^n}^2}\bigg]
\f{1 + \|y\|_{\bbR^n}^2}{\|y\|_{\bbR^n}^2} \, d\nu_{j,k}(y), \quad x \in \bbR^n.   \no 
\end{align}
\end{corollary}
\begin{proof}
The equivalence of items $(iv)$, $(v)$, and $(vi)$ follows from classical results (cf.\ 
\cite[Theorems~XIII.52 and XIII.53]{RS78}). The equivalence of items $(i)$, $(ii)$, $(iii)$, and $(iv)$ 
follows from Theorem \ref{t2.2}, putting $G = \exp_H(tF)$, $t > 0$. 
\end{proof}

\begin{corollary} \lb{c2.4}
Let $m, n \in \bbN$, suppose that $F \colon \bbR^n \to \bbC^{m \times m}$ is a continuous,  
diagonal, matrix-valued function, and assume there exists $c \in \bbR$ such that 
\begin{equation}
\Re(F_{j,j}) \leq c, \quad 1 \leq j \leq m.
\end{equation}
Then the following items $(i)$--$(iv)$ are equivalent: \\[1mm] 
$(i)$ For all $1 \leq j \leq m$, $F_{j,j}$ is conditionally positive semidefinite. \\[1mm]
$(ii)$ For all $t > 0$, the linear operator
\begin{equation}
(\exp(tF))(- i \nabla)\big|_{C_0^{\infty}(\bbR^n)^m} \colon \begin{cases} 
C_0^{\infty}(\bbR^n)^m \to C_{\infty}(\bbR^n)^m, \\
(f_1,\dots,f_m)^{\top} \mapsto \Big(\exp(tF) \big(f_1^{\wedge}, \dots, f_m^{\wedge}\big)^{\top}\Big)^{\vee},
\end{cases}    \lb{2.34}
\end{equation}
extends boundedly to $(\exp(tF))(- i \nabla) \in \cB\big(L^1(\bbR^n)^m\big)$ satisfying 
\begin{equation}
(\exp(tF))(- i \nabla) \big(L^1(\bbR^n)_+^m\big) \subseteq L^1(\bbR^n)_+^m.  
\end{equation}
$(iii)$ There exists $p \in (1,\infty)$ such that for all $t > 0$, the linear operator \eqref{2.34} 
extends boundedly to $(\exp(tF))(- i \nabla) \in \cB\big(L^p(\bbR^n)^m\big)$ satisfying 
\begin{equation}
(\exp(tF))(- i \nabla) \big(L^p(\bbR^n)_+^m\big) \subseteq L^p(\bbR^n)_+^m.  
\end{equation}
$(iv)$ For all $p \in (1,\infty)$, and all $t > 0$, the linear operator \eqref{2.34} 
extends boundedly to $(\exp(tF))(- i \nabla) \in \cB\big(L^p(\bbR^n)^m\big)$ satisfying 
\begin{equation}
(\exp(tF))(- i \nabla) \big(L^p(\bbR^n)_+^m\big) \subseteq L^p(\bbR^n)_+^m.  
\end{equation}
\end{corollary}
\begin{proof}
By the assumptions on $F$, $\exp(tF) \colon \bbR^n \to \bbC^{m \times m}$ is a bounded, continuous, diagonal, matrix-valued function whose diagonal entries are 
\begin{equation}
\exp(tF)_{j,j} = \exp(tF_{j,j}), \quad t > 0, \; 1 \leq j \leq m.
\end{equation}
By Theorem \ref{t2.2}, with $G = \exp(t F)$, $t > 0$, it suffices to prove that item $(i)$ is equivalent to 
\\[1mm]
$(i)(\alpha)$ For all $1 \leq j,k \leq m$, and all $t > 0$, $\exp(t F)_{j,k} \colon \bbR^n \to \bbC$ is positive semidefinite. \\[1mm] 
Since $\exp(t F)$ is a diagonal matrix, item $(i)(\alpha)$ is equivalent to \\[1mm]
$(i)(\beta)$ For all $1 \leq j \leq m$, and all $t > 0$, $\exp(t F)_{j,j} = \exp(t F_{j,j})$ is positive semidefinite. \\[1mm] 
However, the equivalence of items $(i)$ and $(i)(\beta)$ follows from the equivalence of items $(ii)$ and $(iii)$ in Theorem \ref{tA.2}. 
\end{proof}

\begin{corollary} \lb{c2.5}
Let $m, n \in \bbN$, and assume that $F \colon \bbR^n \to \bbC^{m \times m}$ is bounded and continuous. Suppose that for all $1 \leq j,k \leq m$, $F_{j,k} \colon \bbR^n \to \bbC$ is positive semidefinite. Then for all $p \in [1,\infty)$, and all $t > 0$, the linear operator
\begin{equation}
(\exp(tF))(- i \nabla)\big|_{C_0^{\infty}(\bbR^n)^m} \colon \begin{cases} 
C_0^{\infty}(\bbR^n)^m \to C_{\infty}(\bbR^n)^m, \\
(f_1,\dots,f_m)^{\top} \mapsto \Big(\exp(tF) \big(f_1^{\wedge}, \dots, f_m^{\wedge}\big)^{\top}\Big)^{\vee},
\end{cases}
\end{equation}
extends boundedly to $(\exp(tF))(- i \nabla) \in \cB\big(L^p(\bbR^n)^m\big)$ satisfying 
\begin{equation}
(\exp(tF))(- i \nabla) \big(L^p(\bbR^n)_+^m\big) \subseteq L^p(\bbR^n)_+^m.  
\end{equation}
\end{corollary}
\begin{proof}
By the hypotheses on $F$, $\exp(t F) \colon \bbR^n \to \bbC^{m \times m}$ is bounded and continuous for all $t > 0$. By Theorem \ref{t2.2}, with $G = \exp(t F)$, $t > 0$, it suffices to prove that for all 
$1 \leq j,k \leq m$ and all $t > 0$, 
$\exp(t F)_{j,k} \colon \bbR^n \to \bbC$ is positive semidefinite. Combining Lemma \ref{l2.1} with an induction argument shows that $\big(F^{\ell}\big)_{j,k} \colon \bbR^n \to \bbC$ is positive semidefinite for all 
$\ell \in \bbN$ and all $1 \leq j,k \leq m$. Thus, also $\exp(t F)_{j,k} 
= \sum_{\ell = 0}^{\infty} \f{t^{\ell}}{\ell !} (F^{\ell})_{j,k}$ is positive semidefinite for all $t > 0$ and all 
$1 \leq j,k \leq m$. 
\end{proof} 

We conclude with an explicit illustration:

\begin{example} \lb{e2.6}
Let $n \in \bbN$, assume that $a \colon \bbR^n \to \bbR$ is continuous and bounded above, 
$b \geq 0$ is constant, and define $F_0 \colon \bbR^n \to \bbC^{2 \times 2}$ by
\begin{equation} 
F_0(x) = \begin{pmatrix} a(x) & b \\ b & a(x)\end{pmatrix}, \quad x \in \bbR^n.
\end{equation}
Then the following items $(i)$--$(iv)$ are equivalent: \\[1mm]
$(i)$ $a$ is conditionally positive semidefinite. \\[1mm]
$(ii)$ $F_0$ is conditionally positive semidefinite in the sense of Mlak \cite{Ml83}, that is, for all $N \in \bbN$, all $x_p \in \bbR^n$, and all $c_p \in \bbC^2$, $1 \leq p \leq N$, satisfying 
$\sum_{p=1}^N c_p = 0$, one has,
\begin{equation}
\sum_{p,q = 1}^N (c_p, F_0(x_p - x_q) c_q)_{\bbC^2} \geq 0.
\end{equation}
$(iii)$ For all $t > 0$, $\exp(t F_0) \colon \bbR^n \to \bbR^{2 \times 2}$ is positive semidefinite in the sense of 
\cite[Definition~2.4]{GP17} $($cf.\ Definition \ref{dA.4}$)$. \\[1mm] 
$(iv)$ For all $t > 0$, the linear operator 
\begin{equation}
(\exp(tF_0))(- i \nabla)\big|_{C_0^{\infty}(\bbR^n)^m} \colon \begin{cases} 
C_0^{\infty}(\bbR^n)^m \to C_{\infty}(\bbR^n)^m, \\
(f_1,\dots,f_m)^{\top} \mapsto \Big(\exp(tF_0) \big(f_1^{\wedge}, \dots, f_m^{\wedge}\big)^{\top}\Big)^{\vee},
\end{cases}     \lb{2.43}
\end{equation}
extends boundedly to $(\exp(tF_0))(- i \nabla) \in \cB\big(L^1(\bbR^n)^m\big)$ satisfying 
\begin{equation}
(\exp(tF_0))(- i \nabla) \big(L^1(\bbR^n)_+^m\big) \subseteq L^1(\bbR^n)_+^m.  
\end{equation}
$(v)$ There exists $p \in (1,\infty)$ such that for all $t > 0$, the linear operator \eqref{2.43} 
extends boundedly to $(\exp(tF_0))(- i \nabla) \in \cB\big(L^p(\bbR^n)^m\big)$ satisfying 
\begin{equation}
(\exp(tF_0))(- i \nabla) \big(L^p(\bbR^n)_+^m\big) \subseteq L^p(\bbR^n)_+^m.  
\end{equation}
$(vi)$ For all $p \in (1,\infty)$, and all $t > 0$, the linear operator \eqref{2.43} 
extends boundedly to $(\exp(tF_0))(- i \nabla) \in \cB\big(L^p(\bbR^n)^m\big)$ satisfying 
\begin{equation}
(\exp(tF_0))(- i \nabla) \big(L^p(\bbR^n)_+^m\big) \subseteq L^p(\bbR^n)_+^m.  
\end{equation}
\end{example}
\begin{proof}
We start by proving the equivalence of items $(i)$ and $(ii)$. One notes that for 
$c_p = (c_{p,1}, c_{p,2})^{\top} \in \bbC^2$, $1 \leq p \leq N$, $N \in \bbN$, 
\begin{align}
& \sum_{p,q = 1}^N (c_p, F_0(x_p - x_q) c_q)_{\bbC^2} 
= \sum_{p,q = 1}^N \ol{c_{p,1}} \, a(x_p - x_q) c_{q,1} 
+ \sum_{p,q = 1}^N \ol{c_{p,2}} \, a(x_p - x_q) c_{q,2}    \no \\
& \quad + b \bigg[\bigg(\sum_{p=1}^N \ol{c_{p,1}}\bigg) \bigg(\sum_{q=1}^N c_{q,2}\bigg)
+ \bigg(\sum_{p=1}^N \ol{c_{p,2}}\bigg) \bigg(\sum_{q=1}^N c_{q,1}\bigg)\bigg].   \lb{2.47} 
\end{align} 
If item $(i)$ holds, then for $c_p = (c_{p,1}, c_{p,2})^{\top} \in \bbC^2$, $1 \leq p \leq N$, with 
$\sum_{p=1}^N c_p = 0$, \eqref{2.47} implies 
\begin{align}
& \sum_{p,q}^N (c_p, F_0(x_p - x_q) c_q)_{\bbC^2} 
= \sum_{p,q=1}^N \ol{c_{p,1}} \, a(x_p - x_q) c_{q,1} 
+ \sum_{p,q=1}^N \ol{c_{p,2}} \, a(x_p - x_q) c_{q,2}  \no \\
& \quad \geq 0,
\end{align}
implying item $(ii)$. 

Next, suppose item $(ii)$ holds. Choose $d_p \in \bbC$, $1 \leq p \leq N$, $N \in \bbN$, with 
$\sum_{p=1}^N d_p = 0$. Let $c_p = (c_{p,1}, c_{p,2})^{\top} \in \bbC^2$, $1 \leq p \leq N$, be defined via
\begin{equation}
c_{p,1} = d_p, \quad c_{p,2} = 0, \quad 1 \leq p \leq N.
\end{equation}
Then $\sum_{p=1}^N c_p = 0$ and \eqref{2.47} yields 
\begin{equation}
0 \leq \sum_{p,q=1}^N (c_p, F_0(x_p - x_q) c_q)_{\bbC^2} 
= \sum_{p,q=1}^N \ol{d_p} \, a(x_p - x_q) d_q, 
\end{equation}
implying item $(i)$.

Next we employ the elementary matrix identity 
\begin{equation}
\exp\bigg(\begin{pmatrix} \alpha & \beta \\ \beta & \alpha \end{pmatrix}\bigg) 
= e^{\alpha} \begin{pmatrix} \cosh(\beta) & \sinh(\beta) \\ \sinh(\beta) & \cosh(\beta) \end{pmatrix}, 
\quad \alpha, \beta \in \bbR.    \lb{2.51} 
\end{equation}
Then, if $b=0$, the equivalence of items $(i)$, $(iv)$, $(v)$, and $(vi)$ follows from Corollary \ref{c2.4}. 

Next, suppose item $(i)$ holds and that $b =0$. We will show that then also item $(iii)$ holds: Let 
$x_p \in \bbR^n$, and let $c_p = (c_{p,1}, c_{p,2})^{\top} \in \bbC^2$, $1 \leq p \leq N$, $N \in \bbN$. 
Then  
\begin{equation}
\exp(t F_0) (x) = \begin{pmatrix} e^{t a(x)} & 0 \\ 0 & e^{t a(x)} \end{pmatrix}, \quad t > 0, \; x \in \bbR^n,
\end{equation} 
and hence by the the equivalence of items $(ii)$ and $(iii)$ in 
Theorem \ref{tA.2}, 
\begin{align}
& \sum_{p,q=1}^N (c_p, \exp(t F_0)(x_p - x_q) c_q)_{\bbC^2} 
= \sum_{p,q=1}^N \ol{c_{p,1}} \, e^{t a(x_p - x_q)} c_{q,1} +
\sum_{p,q=1}^N \ol{c_{p,2}} \, e^{t a(x_p - x_q)} c_{q,2}    \no \\
& \quad \geq 0.  
\end{align} 
Hence, item $(iii)$ follows from \cite[Lemma~2.5\,$(i)$]{GP17} (cf.\ 
Remark \ref{rA.5}). 

Now suppose item $(iii)$ holds and that $b=0$. We will show that then also item $(ii)$ holds: Let 
$x_p \in \bbR^n$, and let $c_p \in \bbC^2$, $1 \leq p \leq N$, $N \in \bbN$, with $\sum_{p=1}^N c_p = 0$.
Then (with $I_2$ the identity matrix in $\bbC^{2 \times 2}$),
\begin{align}
0 &\leq \lim_{t \downarrow 0} t^{-1} \big(c_p, \exp(t F_0)(x_p - x_q) c_q\big)_{\bbC^2}   \no \\
& = \lim_{t \downarrow 0} \sum_{p,q=1}^N 
\big(c_p, t^{-1}\big[\exp(t F_0)(x_p - x_q) - I_2\big]c_q\big)_{\bbC^2}   \no \\
&= \sum_{p,q=1}^N (c_p, F_0(x_p - x_q) c_q)_{\bbC^2},   \lb{2.54} 
\end{align}
and hence items $(i)$--$(vi)$ are equivalent if $b=0$. 

In the remainder of the proof we suppose that $b \neq 0$. We start by proving that item $(i)$ implies 
item $(iii)$. By inspection, the following matrix is nonnegative,  
\begin{equation}
\begin{pmatrix} \cosh(tb) & \sinh(tb) \\ \sinh(tb) & \cosh(tb) \end{pmatrix} \geq 0,
\end{equation}
as its eigenvalues $\cosh(tb) \pm \sinh(tb)$ are nonnegative. By item $(i)$ and the equivalence of items $(ii)$ and $(iii)$ in Theorem \ref{tA.2} as well as  Bochner's Theorems \ref{tA.3}), for all $t > 0$, there exists a finite, nonnegative Borel measure $\nu_t$ on $\bbR^n$ such that 
$e^{ta} = \nu_t^{\wedge}$. Thus, \eqref{2.51} implies 
\begin{align}
\exp(t F_0(x)) &= e^{t a(x)} \begin{pmatrix} \cosh(tb) & \sinh(tb) \\ \sinh(tb) & \cosh(tb) \end{pmatrix}  
\no \\
&= \bigg(\begin{pmatrix} \cosh(tb) & \sinh(tb) \\ \sinh(tb) & \cosh(tb) \end{pmatrix}  \nu_t\bigg)^{\wedge},
\lb{2.56} 
\end{align}
that is, $\exp(t F_0(x))$ is the Fourier transform of a nonnegative, finite, $\bbC^{2 \times 2}$-valued measure, and hence item $(iii)$ follows from Berberian's matrix-valued extension of Bochner's theorem, Theorem \ref{tA.6}. 

To verify that item $(iii)$ impies item $(ii)$, one notes that \eqref{2.54} remains valid if $b \neq 0$. 

By equation \eqref{2.56}, for all $t > 0$, the entries of $\exp(t F_0)$ are either $e^{ta} \cosh(tb)$ or 
$e^{ta} \sinh(tb)$, and hence for all $1 \leq j,k \leq 2$, $\exp(t F_0(\, \cdot \,))_{j,k}$ is positive 
semidefinite if and only if $e^{ta(\,\cdot \,)}$ is, and thus, by the equivalence of items $(ii)$ and $(iii)$ in Theorem \ref{tA.2}, if and only if 
$a(\, \cdot \,)$ is conditionally positive semidefinite. The equivalence of items $(i)$, $(iv)$--$(vi)$ 
now follows from Theorem \ref{t2.2}.
\end{proof}

\appendix
\section{Some Background Material} \lb{sA}
\renewcommand{\theequation}{A.\arabic{equation}}
\renewcommand{\thetheorem}{A.\arabic{theorem}}
\setcounter{theorem}{0} \setcounter{equation}{0}

We briefly recall the basic definitions of (conditionally) positive semidefinite 
functions $F \colon \bbR^n \to \bbC$, and state two classical results in this context; we also briefly 
hint at a matrix-valued extension of Bochner's theorem (for details we refer to \cite{GP17} and the extensive literature cited therein). Finally, we recall some results on Fourier multipliers in $L^1(\bbR^n)$ and $L^2(\bbR^n)$ (see \cite[Sect.~2.5]{Gr08} for a detailed exposition). 

\begin{definition} \lb{dA.1}
Let $m \in \bbN$, and $A \in \bbC^{m \times m}$, and 
suppose that $F \colon \bbR^n \to \bbC$, $n \in \bbN$. \\[1mm]
$(i)$ $A$ is called positive semidefinite, also denoted by $A \geq 0$, if 
\begin{equation}
(c, A c)_{\bbC^m} = \sum_{j,k=1}^m \ol{c_j} \, A_{j,k} c_k \geq 0 \, 
\text{ for all } \, c=(c_1,\dots,c_m)^{\top} \in \bbC^m. 
\end{equation} 
$(ii)$ $A=(A_{j,k})_{1\leq j,k \leq m} = A^* \in \bbC^{m \times m}$ is said to 
be conditionally positive semidefinite if 
\begin{equation}
(c, A c)_{\bbC^m} \geq 0 \, 
\text{ for all } \, c=(c_1,\dots,c_m)^{\top} \in \bbC^m, \, \text{ with } \, \sum_{j=1}^m c_j = 0. 
\end{equation} 
$(iii)$ $F$ is called positive semidefinite if for all $N \in \bbN$, $x_p \in \bbR^n$, 
$1 \leq p \leq N$, the matrix $(F(x_p - x_q))_{1 \leq p,q \leq N} \in \bbC^{N \times N}$ is 
positive semidefinite. \\[1mm] 
$(iv)$ $F$ is called conditionally positive semidefinite if for all $N \in \bbN$, $x_p \in \bbR^n$, 
$1 \leq p \leq N$, the 
matrix $(F(x_p - x_q))_{1 \leq p,q \leq N} \in \bbC^{N \times N}$ is conditionally positive 
semidefinite. \\[1mm]
$(v)$ Let $T \in \cB\big( L^2(\bbR^n)\big)$.~Then $T$ is called positivity preserving  
$($in $ L^2(\bbR^n)$$)$ if for any $0 \leq f \in L^2(\bbR^n)$ also $T f \geq 0$. \\[1mm]
\end{definition} 

In connection with Definition \ref{dA.1}\,$(iv)$ one can show that if $F$ is conditionally positive semidefinite, then (cf.\ \cite[Lemma~2.5\,$(iii)$]{GP17}, 
\cite[p.~12]{RS75})
\begin{equation} 
F(-x) = \ol{F(x)}, \quad x \in \bbR^n.      \lb{A.3} 
\end{equation} 
In addition, one observes that for $T$ to be positivity preserving it suffices to take $0 \leq f \in C_0^{\infty} (\bbR^n)$ in Definition \ref{dA.1}\,$(v)$. 

Given $F \in C(\bbR^n)$ and $F$ polynomially bounded, one can define 
\begin{equation}
F(- i \nabla) \colon \begin{cases} C_0^{\infty}(\bbR^n) \to L^2(\bbR^n),  \\ 
f \mapsto F(- i \nabla) f = \big(f^{\wedge} F\big)^{\vee}. \end{cases} 
\end{equation}
More generally, if $F \in L^1_{\loc}(\bbR^n)$, one introduces the maximally defined operator of 
multiplication by $F$ in $L^2(\bbR^n)$, denoted by $M_F$, by
\begin{equation}
(M_F f)(x) = F(x) f(x), \quad 
f \in \dom(M_F) = \big\{g \in L^2(\bbR^n) \, \big | \, F g \in L^2(\bbR^n)\big\}, 
\end{equation} 
and then defines $F(- i \nabla)$ as a normal operator in $L^2(\bbR^n)$ via
\begin{equation} 
F(- i \nabla) = \cF^{-1} M_F \cF 
\end{equation}
(cf.\ \eqref{1.15}, \eqref{1.16} and their unitary extensions to $L^2(\bbR^n)$).  

\begin{theorem} [cf., e.g., {\cite{HS78}}, {\cite{JS98}}, {\cite[Theorems~XIII.52 and XIII.53]{RS78}, {\cite{Sc38}}}] \lb{tA.2} ${}$ \\[1mm] 
Assume that $F \in C(\bbR^n)$ and there exists $c \in \bbR$ such that $\Re(F(x)) \leq c$. Then the following items $(i)$--$(iv)$ are equivalent: \\[1mm] 
$(i)$ For all $t > 0$, $\exp(t F(-i \nabla))$ is positivity preserving in $L^2(\bbR^n)$. \\[1mm] 
$(ii)$ For each $t > 0$, $e^{t F}$ is a positive semidefinite function. \\[1mm] 
$(iii)$ $F$ is conditionally positive semidefinite. \\[1mm]
$(iv)$ $($The Levy--Khintchine formula\,$)$. There exists, $\alpha \in \bbR$, $\beta \in \bbR^n$, 
$0 \leq A \in \bbC^{n \times n}$, and a nonnegative finite measure $\nu$ on $\bbR^n$, with 
$\nu(\{0\}) = 0$, such that 
\begin{align}
\begin{split} 
F(x) &= \alpha + i (\beta, x)_{\bbR^n} - (x, A x)_{\bbC^n}   \\ 
& \quad + \int_{\bbR^n} \bigg[\exp(i (x, y)_{\bbR^n}) -1 - 
\frac{i (x, y)_{\bbR^n}}{1 + \|y\|_{\bbR^n}^2} \bigg] \f{1 + \|y\|_{\bbR^n}^2}{\|y\|_{\bbR^n}^2} \, d \nu(y),  
\quad x \in \bbR^n.
\end{split} 
\end{align}
\end{theorem}

Just for completeness (as it is used repeatedly in the bulk of this paper), we recall that item $(ii)$ above implies item $(iii)$ by differentiating 
$exp_H\big(t(F(x_p-x_q))_{1\leq p,q \leq N}\big)$, $x_p \in \bbR^n$, $1\leq p \leq N$, $N \in \bbN$, at $t = 0$.  Conversely, that item $(iii)$ implies item $(ii)$ is a consequence of \cite[Theorem~6.3.6]{HJ94}.

We continue by recalling Bochner's classical theorem \cite{Bo33}:

\begin{theorem} $($Bochner's Theorem, cf., e.g., {\cite[Sect.~5.4]{Ak65}}, {\cite[Theorem~2.7]{JS98}},
{\cite[p.~13]{RS75}}, {\cite[p.~46]{SSV12}}$)$. \lb{tA.3} 
${}$ \\ 
Assume that $F \in C(\bbR^n)$. Then the following items $(i)$ and 
$(ii)$ are equivalent: \\[1mm] 
$(i)$ $F$ is positive semidefinite. \\[1mm]
$(ii)$ There exists a nonnegative finite measure $\mu$ on $\bbR^n$ such that 
\begin{equation}
 F(x) = \mu^{\wedge}(x), \quad x \in \bbR^n.    \lb{A.8} 
\end{equation} 
In addition, if one of conditions $(i)$ or $(ii)$ holds, then
\begin{equation}
F(-x) = \ol{F(x)}, \quad |F(x)| \leq |F(0)|, \quad x \in \bbR^n, 
\end{equation} 
in particular, $F$ is bounded on $\bbR^n$. 
\end{theorem}

Next, we turn to the finite-dimensional special case of an infinite-dimensional extension of Bochner's theorem in connection with locally compact Abelian  groups due to Berberian \cite{Be66} (see also \cite{FH72}, \cite{FK91}, \cite{Ml83}, \cite{vW68}):

\begin{definition} \lb{dA.4}
Let $F \colon \bbR^n \to \bbC^{m \times m}$, $m, n \in \bbN$. Then $F$ is called 
positive semidefinite if for all $N \in \bbN$, $x_p \in \bbR^n$, 
$1 \leq p \leq N$, the block matrix $(F(x_p - x_q))_{1 \leq p,q \leq N} \in \bbC^{mN \times mN}$ is 
nonnegative.
\end{definition}

\begin{remark} \lb{rA.5}
By \cite[Lemma~2.5\,$(i)$]{GP17}, $F\colon \bbR^n \to \bbC^{m \times m}$ is positive semidefinite if and only if for all $N \in \bbN$, $x_p \in \bbR^n$, $c_p \in \bbC^m$, 1$ \leq p \leq N$, one has 
\begin{equation}
\sum_{p,q = 1}^N (c_p, F(x_p - x_q) c_q)_{\bbC^m} \geq  0. 
\end{equation}
\end{remark}

\begin{theorem} [{\cite[p~178, Theorem~3 and Corollary on p.~177]{Be66}}] \lb{tA.6} ${}$ \\[1mm]
Assume that $F \in C(\bbR^n, \bbC^{m \times m}) \cap L^{\infty} (\bbR^n, \bbC^{m \times m})$, 
$m \in \bbN$. Then the following items $(i)$ and $(ii)$ are equivalent: \\[1mm]
$(i)$ $F$ is positive semidefinite. \\[1mm]
$(ii)$ There exists a nonnegative measure $\mu \in \cM(\bbR^n, \bbC^{m \times m})$ such that 
\begin{equation}
 F(x) = \mu^{\wedge}(x), \quad x \in \bbR^n.    \lb{A.7} 
\end{equation} 
In addition, if one of conditions $(i)$ or $(ii)$ holds, then
\begin{equation}
F(-x) = F(x)^*, \quad \|F(x)\|_{\cB(\bbC^m)} \leq \|F(0)\|_{\cB(\bbC^m)}, \quad x \in \bbR^n. 
\lb{A.11} 
\end{equation} 
\end{theorem}

Finally, we briefly turn to Fourier multipliers.

\begin{definition} \lb{dA.6}
Let $p, q \in [1,\infty) \cup \{\infty\}$ The set $\cM^{p,q}(\bbR^n)$ denotes the Banach space of 
all bounded linear operators from $L^p(\bbR^n)$ to $L^q(\bbR^n)$ that commute with translations. The  
norm of $T \in \cM^{p,q}(\bbR^n)$ is given by the operator norm, 
\begin{equation}
\|T\|_{p,q} := \|T\|_{\cB(L^p(\bbR^n), L^q(\bbR^n))}.
\end{equation}  
\end{definition}

We note that bounded convolution operators from $L^p(\bbR^n)$ to $L^q(\bbR^n)$ clearly are commuting with translations (i.e., are translation invariant); that the converse is valid as well is proved 
in \cite[Theorem~2.5.2]{Gr08}.

Given a complex measure $\mu$ on $\bbR^n$, the total variation of $\mu$ is defined by $|\mu|(\bbR^n)$ and the norm of $\mu$ is introduced as $\|\mu\| = |\mu|(\bbR^n)$. Given a $\mu$-integrable $f \colon \bbR^n \to \bbC$, one defines the convolution of $f$ and $\mu$ by  
\begin{equation}
f * \mu \colon \begin{cases} \bbR^n \to \bbC, \\ 
x \mapsto (f * \mu)(x) = \int_{\bbR^n} f(x-y) \, d\mu(y), 
\end{cases} \quad x \in \bbR^n.    \lb{A.9}
\end{equation}
In addition, one can introduce the associated convolution operator 
$T_{\mu} \in \cB\big( L^p(\bbR^n)\big)$, $p \in [1,\infty)$,  by
\begin{equation}
T_{\mu} f = f * \mu, \quad f \in L^p(\bbR^n).
\end{equation}
Similarly, if $u \in \cS'(\bbR^n)$ is a tempered distribution, the associated convolution operator  
$T_u$ is defined via
\begin{equation}
T_{u} f = f * u, \quad f \in \cS(\bbR^n).
\end{equation}

In the cases $p=q = 1,2$ one has the following well-known results:

\begin{theorem} 
[cf., e.g., {\cite[Theorem~2.5.8 and Theorem~2.5.10]{Gr08}}] \lb{tA.7} ${}$ \\[1mm] 
$(i)$ $T \in \cM^{1,1}(\bbR^n)$ if and only if $T = T_{\mu}$ for some $($finite$)$ complex measure $\mu$. In this case 
\begin{equation}
\|T\|_{1,1} = \|T_{\mu}\|_{\cB(L^1(\bbR^n))} = |\mu|(\bbR^n). 
\end{equation} 
$(ii)$ $T \in \cM^{2,2}(\bbR^n)$ if and only if $T = T_{u}$ for some $u  \in \cS'(\bbR^n)$, whose Fourier transform $u^{\wedge}$ lies in $L^{\infty}(\bbR^n)$. In this case 
\begin{equation}
\|T\|_{2,2} = \|T_u\|_{\cB(L^2(\bbR^n))} = \big\|u^{\wedge}\big\|_{L^{\infty}(\bbR^n)}. 
\end{equation} 
\end{theorem}

\medskip

\noindent
{\bf Acknowledgments.} We are indebted to Tao Mei for very stimulating discussions. 
 
 
\end{document}